\documentclass{amsart}
\usepackage{latexsym,rotate,eucal,cite}
\usepackage{amsmath,amsthm,amssymb,amsxtra}
\usepackage[pagewise]{lineno}
\usepackage{cite}
\newtheorem{theorem}{Theorem}[section]
\newtheorem{lemma}[theorem]{Lemma}
\newtheorem{proposition}[theorem]{Proposition}

\theoremstyle{definition}

\theoremstyle{remark}
\newtheorem{remark}[theorem]{Remark}

\numberwithin{equation}{section}

\begin{document}

\title{On Chen's biharmonic conjecture for hypersurfaces in $\mathbb R^5$}

\author{Yu Fu}
\address{School of Mathematics, Dongbei University of Finance and Economics,
Dalian 116025, P. R. China}\email{yufu@dufe.edu.cn}


\author{Min-Chun Hong}
\address{Department of Mathematics, The University of Queensland, Brisbane,
QLD 4072, Australia} \email{hong@maths.uq.edu.au}

\author{Xin Zhan}
\address{School of Mathematics Science, Dalian University of Technology,
Dalian 116023, P. R. China} \email{zhanxin@mail.dlut.edu.cn}





\begin{abstract}
 A longstanding conjecture on biharmonic submanifolds, proposed by Chen
in 1991, is that {\it any biharmonic submanifold in a Euclidean
space is minimal}. In the case of a hypersurface $M^n$ in $\mathbb
R^{n+1}$, Chen's conjecture was settled in the case of $n=2$ by Chen
and Jiang around 1987 independently.  Hasanis and Vlachos in 1995
settled Chen's conjecture for a hypersurface with $n=3$. However,
the general  Chen's conjecture on a hypersurface $M^n$ remains open
for $n>3$. In this paper, we  settle Chen's conjecture for
hypersurfaces in $\mathbb R^{5}$ for  $n=4$.
\end{abstract}

\maketitle \markboth{Fu, Hong and Zhan} {On Chen's biharmonic
conjecture for hypersurfaces }

\section{Introduction}
 In 1983,  Eells and  Lemaire \cite{Eells} introduced the  concept of biharmonic
maps between Riemannian manifolds. Since then,  biharmonic maps  have
been extensively studied by mathematicians. In particular, geometers
investigated a special class of biharmonic maps named {\em
biharmonic immersions}. An immersion
$\phi:(M^n,g)\longrightarrow(N^m,h)$ is biharmonic  if and only if
its mean curvature vector field $\overrightarrow{H}$ satisfies the
fourth-order semi-linear elliptic equation  (e.g.
\cite{CMO2001},\cite{Oniciuc2002}, \cite{Oniciuc2012})
\begin{eqnarray} \label{biharmonic1}
\Delta\overrightarrow{H}+{\rm trace}\,
R^{N}(d\phi,\overrightarrow{H})d\phi=0,
\end{eqnarray}
where $\Delta$ is the Laplacian of $M$, $\overrightarrow H$ is the
mean curvature vector of the immersion and $R^N$ is the curvature
tensor of $N$.

Biharmonic submanifolds have attracted a lot of attentions  and many  results on biharmonic submanifolds
were obtained in  past two decades (e.g. [3-5, 8, 16-30, 32,
35, 38]).

In the Euclidean ambient space, biharmonic submanifolds are defined
via the geometric condition $\Delta \overrightarrow{H}=0$, or
equivalently $\Delta^2 \phi=0$, which was originally proposed by B.
Y. Chen in his pioneering work of finite type theory in the middle
of 1980s (c.f. \cite{Chen1991}, \cite{chenbook2015}). It is easy to
see from the definition that minimal submanifolds are automatically
biharmonic.

B. Y. Chen \cite{Chen1991} in 1991 proposed a well-known conjecture
in the following:

{\bf Chen's conjecture}: {\em Any biharmonic submanifold in the
Euclidean space $\mathbb R^m$ is minimal}.

Although partial results on Chen's conjecture were obtained  for low
dimensions and with additional geometric conditions (e.g.
\cite{akutagawa2013}, \cite{luoyong2014}, \cite{Ou-Chen-book2020},
\cite{Ou2016}), Chen's conjecture is widely open.

Since  the most important  biharmonic submanifolds are hypersurfaces
in Euclidean spaces,  Chen's conjecture on  hypersurfaces in
Euclidean spaces  is a basic one and has been investigated by many
mathematicians. Let $M^n$ be a biharmonic hypersurface in the
Euclidean space $\mathbb R^{n+1}$. Chen in 1986 and Jiang
\cite{jiang1987} in 1987 independently made  a pioneering
contribution to prove that a biharmonic surface $M^2$ in $\mathbb
R^3$ is minimal. Dimitri\'{c} \cite{Dimitri1992}   extended this
result to hypersurfaces with at most two distinct principal
curvatures in $\mathbb R^{n+1}$. In 1995,  Hasanis and Vlachos
\cite{HasanisVlachos1995} made an important progress and  settled
the conjecture for the case $n=3$. Defever  \cite{defever1998} in
1998 reproved the conjecture for $n=3$ through  a new tensorial
analysis approach. Later, Chen and Munteanu \cite{chenMunteanu2013}
confirmed Chen's conjecture for $\delta(2)$-ideal and
$\delta(3)$-ideal hypersurfaces in $\mathbb R^{n+1}$. The first
author \cite{fu1} in 2015 showed that Chen's conjecture is true for
hypersurfaces with three distinct principal curvatures in $\mathbb
R^{n+1}$. Recently, Montaldo, Oniciuc and Ratto \cite{montaldo}
confirmed the conjecture for the $G-$invariant hypersurfaces of
cohomogeneity one in $\mathbb R^{n+1}$.  Koiso and Urakawa
\cite{KoisoUrakawa2014} also verified the conjecture for generic
hypersufaces with irreducible principal curvature vector fields in
$\mathbb R^{n+1}$. However, Chen's conjecture remains open until now
for hypersurfaces $M^n$ for $n\geq 4$.

It is  known that every hypersurface is locally a graph. The
famous Bernstein problem for minimal graphs in $\Bbb R^n$ states
that for $n\geq2$, any entire solution $f:\mathbb R^{n}\to \mathbb
R$ of the minimal graph equation
\begin{equation}
\sum_{i,j=1}^n\Big(\delta_{ij}-\frac{f_if_j}{1+|\nabla
f|^2}\Big)f_{ij}=0\nonumber
\end{equation}
is  an affine function.

The Bernstein problem was first investigated by Bernstein in  1915 for
$n=2$, solved by    De Giorgi \cite{DeGiorgi1965} in  1965  for $n=3$ and
Almgren \cite{Almgren1966}  in  1966   for $n=4$. Simons
\cite{Simon-1968} in 1968 solved it for $n\leq7$. Finally,
Bombieri-De Giorgi-Giusti  \cite{Bombieri1969} in 1969 affirmatively
resolved the  Bernstein problem by constructing a counterexample
that the Bernstein problem is not true for $n\geq 8$.

There exists an interesting similarity between the Bernstein problem
and Chen's conjecture in the case of the graph in $\Bbb R^n$. Chen's
conjecture for biharmonic graphs states that for $n\geq2$, any
entire solution $f:\mathbb R^{n}\to \mathbb R$ of the biharmonic
graph equations (c.f. \cite{Ou20121})
\begin{equation}
\begin{cases}
\Delta (\Delta f)=0,\nonumber\\
(\Delta f_k)\Delta f+2\left <\nabla f_k,\nabla\Delta f\right  >=0,
\quad k=1, \ldots n \nonumber
\end{cases}
\end{equation}
is minimal; i.e.  $\Delta f=-\mathrm {div}(\nabla f)=0$.

In contrast to the Bernstein problem, Chen's conjecture holds true
for $n=2$ (Chen \cite{Chen1991}, Jiang 1987 \cite{jiang1987}) and
for $n=3$ (Hasanis-Vlachos, 1995 \cite{HasanisVlachos1995}).
However,  Chen's conjecture for the case $n\geq4$  is unknown even
for the biharmonic graphs. By comparing with the Bernstein problem,
it is very interesting and important to investigate  Chen's
conjecture for the case of $n\geq 4$.

In this paper, we confirm Chen's conjecture  for hypersurfaces in
the case of $n=4$.  More precisely, we prove
\begin{theorem}
Every biharmonic hypersurface in the Euclidean space $\mathbb R^{5}$
is minimal.
\end{theorem}

 The main approach for the proof of Theorem 1.1 is a
continuation of the program developed in \cite{fu-hong2018}.
However, in \cite{fu-hong2018}, an extra condition of {\em constant
scalar curvature} is assumed. In this paper, we overcome the
difficulty and remove the condition for $n=4$ by  dealing with the
differential equations related to biharmonicity. By transferring the
biharmonic equations into a system of algebraic differential
equations, we developed a new method to determine the behavior of
the principal curvature functions via investigating the solution of
the system of algebraic differential equations. This new approach in
the current paper provides an interesting insight for us to
understand the geometric structure of biharmonic hypersurfaces.

We would like to outline our proofs here. Assume that the mean
curvature $H$ is non-constant. Through  Proposition 2.1, the
immersion $\phi: M^n\rightarrow\mathbb{R}^{n+1}$ of a biharmonic
hypersurface $M^n$ in  $\mathbb R^{n+1}$ satisfies two geometric
equations \eqref{biharmonic condition}. The second one of
\eqref{biharmonic condition} tells us that ${\rm grad}\,H$ is an
eigenvector of the Weingarten operator $A$ with the corresponding
principal curvature $\lambda_1=-2H$. For the non-constant  mean
curvature $H$,  the multiplicity of the principal curvature
$\lambda_1$  is one (see Lemma 2.2). By using the Gauss and Codazzi
equations, biharmonic equations reduce to three differential
equations \eqref{L1}-\eqref{L3} on principal curvature functions
$\lambda_i$. Noticing that the equations in Lemma 2.3 are an
over-determined system of differential equations related to the
principal curvatures and the coefficients of connection, we transfer
the equations \eqref{L1}-\eqref{L3} into five equations concerning
$\sum_{i=2}^4(\omega_{ii}^1)^k$ in terms of $\lambda_1$ and $T$ (see
details in Section 3). In the case of $n=4$, we derive two equations
concerning $T$ and $\lambda_1$ without $\omega_{ii}^1$. Utilizing
the elimination method via differential equations, we show that $T$
is also a smooth function depending only on one variable $t$.
Through relations between $\lambda_i$ and $\omega_{ij}^k$ in Lemmas
4.1 and 4.2, we distinguish $\omega_{ij}^k$ in two cases. For each
case,  we deal with different algebraic equations by a complicated
deduction and get a contradiction finally.

\begin{remark}
It should be pointed out that   Chen's conjecture
for hypersurfaces in the case of $n>4$  is still challenging and the geometric properties
of biharmonic hypersurfaces are far from our understanding and
handling.\end{remark}

The paper is organized as follows. In Section 2, we recall some
necessary background for theory of biharmonic immersions and two
useful lemmas. In Section 3, we derive some key lemmas, which are
crucial to prove the main theorem. In Section 4, we complete a proof
of Theorem 1.1.

\medskip\noindent
{\bf Acknowledgement:} {The authors would like to thank Professors
Bang-Yen Chen and Cezar Oniciuc for their interest and useful
comments. The first author is supported by Liaoning Provincial
Science and Technology Department Project (No.2020-MS-340), and
Liaoning BaiQianWan Talents Program. The partial research of the second author was supported
by the Australian Research Council grant (DP150101275).}

\section{Preliminaries}

A biharmonic map $\phi$  between an $n$-dimensional Riemannian
manifold  $(M^n,g)$ and an $m$-dimensional Riemannian manifold
$(N^m,h)$ is  a critical point of the bienergy functional
\begin{eqnarray*}
E_2(\phi)=\frac{1}{2}\int_M|\tau(\phi)|^2dv_g,
\end{eqnarray*}
where $\tau(\phi)= {\rm trace \nabla d\phi}$ is the tension field of
$\phi$ that vanishes for a harmonic map. The concept of biharmonic
maps was introduced in 1983 by Eells and  Lemaire \cite{Eells} with
the aim to study $k-$harmonic maps. The Euler-Lagrange equation
associated to the bi-energy is stated as
\begin{align*}
\tau_2(\phi)&=-\Delta\tau(\phi)-{\rm trace}\,
R^{N}(d\phi,\tau(\phi))d\phi=0,
\end{align*}
where $\tau_2(\phi)$ is  the bitension field of $\phi$, and $R^{N}$
is the curvature tensor of $N^m$
(e.g.\cite{jiang1986},\cite{jiang1987}). Hence, $\phi$ is called
{\em a biharmonic map} if its bitension field $\tau_2(\phi)$
vanishes identically.

In a special case, an immersion $\phi:(M^n,g)\longrightarrow(N^m,h)$
is biharmonic if and only if its mean curvature vector field
$\overrightarrow{H}$ fulfills the fourth-order semi-linear elliptic
equation \eqref{biharmonic1}. It is well-known that any minimal
immersion is harmonic. The biharmonic immersions are
called proper biharmonic if they are not harmonic.

Let $\phi: M^n\rightarrow\mathbb{R}^{n+1}$ be an isometric immersion
of a hypersurface $M^n$ in the Euclidean space $\mathbb{R}^{n+1}$ and let
   $X$, $Y$, $Z$ be  the  tangent vector fields of
$M^n$. Denote the Levi-Civita connections of $M^n$ and
$\mathbb{R}^{n+1}$ by $\nabla$ and $\tilde\nabla$  respectively. Then the
Gauss and Codazzi equations are written as
\begin{eqnarray*}
R(X,Y)Z=\langle AY,Z\rangle AX-\langle AX,Z\rangle AY,
\end{eqnarray*}
\begin{eqnarray*}
(\nabla_{X} A)Y=(\nabla_{Y} A)X,
\end{eqnarray*}
where $A$ is the Weingarten operator, and $R$ is the curvature
tensor of $M^n$.

We recall
that the mean curvature vector field $\overrightarrow{H}$ can be
defined by
\begin{eqnarray}\label{md}
\overrightarrow{H}=\frac{1}{n}{\rm trace}~h,
\end{eqnarray}
where $h$ is the second fundamental form. For  the mean curvature $H$,
choose $\xi$ to be  the unit normal vector field of $M^n$  satisfying
$\overrightarrow{H}=H\xi$.

The sufficient and necessary conditions
for a hypersurface $M^n$ to be biharmonic are given (see
\cite{CMO2002}, \cite{chenbook2015}), which are the basic
characterizations of a biharmonic hypersurface in $\mathbb R^{n+1}$  in the following:
\begin{proposition}
The immersion $\phi: M^n\rightarrow\mathbb{R}^{n+1}$ of a
hypersurface $M^n$ in the Euclidean space $\mathbb
R^{n+1}$ is biharmonic if and only if $H$ and $A$ satisfy
\begin{equation} \label{biharmonic condition}
\begin{cases}
\Delta H+H {\rm trace}\, A^2 =0,\\
2A\,\nabla H+nH\nabla H=0,
\end{cases}
\end{equation}
\end{proposition}
\noindent where the Laplacian operator $\Delta$ applied on a function $f$ is
given by
\begin{eqnarray*}
\Delta f=-{\rm div}(\nabla
f)=-\sum_{i=1}^n\langle\nabla_{e_i}(\nabla f), e_i
\rangle=-\sum_{i=1}^n(e_ie_i-\nabla_{e_i}e_i)f.
\end{eqnarray*}

 We collect two results from \cite{fu-hong2018} concerning
principal curvatures for biharmonic hypersurfaces, whose proofs are
standard (see also \cite{KoisoUrakawa2014}).
\begin{lemma}{\rm([20])}
Let $M^n$ be an orientable biharmonic hypersurface with non-constant
mean curvature in $\mathbb R^{n+1}$  and  assume the mean curvature $H$
is non-constant. Then the multiplicity of the principal curvature
$\lambda_1$ $(=-nH/2)$ is one, i.e. $\lambda_j\neq\lambda_1$ for
$2\leq j\leq n$.
\end{lemma}
Using the Gauss and Codazzi equations, biharmonic equations
\eqref{biharmonic condition} can be summarized into a system of $2n-1$
differential equations as follows:
\begin{lemma} {\rm([20],[25])} Assume that $H$
is non-constant. Then the smooth real-valued principal curvature
functions $\lambda_i$ and the coefficients of connection
$\omega_{ii}^1$ $(i=2,\ldots, n)$ satisfy the following
differential equations
\begin{align}
&e_1e_1(\lambda_1)=e_1(\lambda_1)\Big(\sum_{i=2}^n\omega_{ii}^1\Big)+\lambda_1S,\label{L1}\\
&e_1(\lambda_i)=\lambda_i\omega_{ii}^1-\lambda_1\omega_{ii}^1,\label{L2}\\
&e_1(\omega_{ii}^1)=(\omega_{ii}^1)^2+\lambda_1\lambda_i,\label{L3}
\end{align}
where $\lambda_1=-nH/2$, $e_1={\nabla}H/|{\nabla} H|$  and $S$ is
the squared length of the second fundamental form $h$ of $M$.
\end{lemma}
\section{Some key lemmas}
From now on, we study the biharmonicity of a hypersurface $M^n$ in a
Euclidean space $\mathbb R^{n+1}$ with $n=4$. Since $M^4$ with at
most three distinct principal curvatures everywhere in a Euclidean
space $\mathbb R^5$ is minimal (see \cite{fu1}), we work only on the
case that the connected component of $M_A$ has different principal
curvatures. In general, the set $M_A$ of all points of $M^4$, at
which the number of distinct eigenvalues of the Weingarten operator
$A$ (i.e. the principal curvatures) is locally a constant, is open
and dense in $M^4$. Meanwhile, on each connected component, the
principal curvature functions of $A$ are always smooth. Assume that,
on a component, the mean curvature $H$ is non-constant. Then there
exists  a neighborhood $U_p$ of $p$ such that $\nabla H \neq0$.

It follows from \eqref{biharmonic condition} that $\nabla H$ is an
eigenvector of the Weingarten operator $A$ with the corresponding
principal curvature $-2H$. Choosing locally $e_1$ such that $e_1$ is
parallel to $\nabla H$, and with respect to some suitable
orthonormal frame $\{e_1,e_2,e_3,e_4\}$, the Weingarten operator $A$
of $M$ is given by
\begin{eqnarray}\label{A}
A=\mathrm{diag}(\lambda_1,\lambda_2, \lambda_3, \lambda_4),\nonumber
\end{eqnarray}
where $\lambda_i$ are the principal curvatures and
$\lambda_1=-2H$. Therefore, it follows from \eqref{md} that
$\sum_{i=1}^4\lambda_i=4H$, and hence
\begin{eqnarray}\label{sum1}
\lambda_2+\lambda_3+\lambda_4=-3\lambda_1.
\end{eqnarray}
Denote by $S$ the squared length of the second fundamental form $h$
of $M$. It follows from \eqref{A} that $S$ is given by
\begin{eqnarray}\label{sum2}
S={\rm trace}\, A^2
=\sum_{i=1}^4\lambda^2_i=\sum_{i=2}^4\lambda^2_i+\lambda^2_1.
\end{eqnarray}
As $\nabla H=\sum_{i=1}^4e_i(H)e_i$ and $e_1$ is parallel to the
direction of $\nabla H$, we have that
\begin{eqnarray*}
e_1(H)\neq0,\quad e_i(H)=0, \quad 2\leq i\leq 4
\end{eqnarray*}
and hence
\begin{eqnarray} \label{H1}
e_1(\lambda_1)\neq0,\quad e_i(\lambda_1)=0, \quad 2\leq i\leq
4.\label{A2}
\end{eqnarray}
Setting $ \nabla_{e_i}e_j=\sum_{k=1}^4\omega_{ij}^ke_k$ $(1\leq
i,j\leq 4)$, a direct computation concerning the compatibility
conditions $\nabla_{e_k}\langle e_i,e_i\rangle=0$ and
$\nabla_{e_k}\langle e_i,e_j\rangle=0$ $(i\neq j)$ yields
respectively that
\begin{eqnarray}
\omega_{ki}^i=0,\quad \omega_{ki}^j+\omega_{kj}^i=0,\quad i\neq
j.\nonumber
\end{eqnarray}
It follows from using the Codazzi equation  that
\begin{eqnarray}
e_i(\lambda_j)=(\lambda_i-\lambda_j)\omega_{ji}^j,\nonumber\\
(\lambda_i-\lambda_j)\omega_{ki}^j=(\lambda_k-\lambda_j)\omega_{ik}^j\nonumber
\end{eqnarray}
for distinct $i, j, k$.

Due to \eqref{H1}, we consider an integral curve of $e_1$
passing through $p=\gamma(t_0)$ as $\gamma(t)$, $t\in I$. It is easy
to show that there exists a local chart $(U; t=x^1,x^2, x^3, x^4)$
around $p$, such that $\lambda_1(t, x^2, x^3, x^4)=\lambda_1(t)$ on
the whole neighborhood of $p$.

Set
$f_k=(\omega_{22}^1)^k+(\omega_{33}^1)^k+(\omega_{44}^1)^k,~~\mathrm
{for}~ k=1,\ldots, 5$. In the following, an interesting system of
algebraic equations will be derived.
\begin{lemma} \label{lemma3.1}
With the notions $f_k$, the following two relations hold
\begin{align}
    &f_1^4-6f_1^2f_2+3f_2^2+8f_1f_3-6f_4=0, \label{L4}\\
    &f_1^5-5f_1^3f_2+5f_1^2f_3+5f_2f_3-6f_5=0.\label{L5}
\end{align}
\end{lemma}

\begin{proof}
It is easy to check that
\begin{equation}\label{P1}
    f_1^2-f_2=\big(\sum_{i=2}^4\omega_{ii}^1\big)^2-\sum_{i=2}^4(\omega_{ii}^1)^2\\
=2(\omega_{22}^1\omega_{33}^1+\omega_{22}^1\omega_{44}^1+\omega_{33}^1\omega_{44}^1)
\end{equation}
and
\begin{equation}\label{P2}
    f_2^2-f_4=\big(\sum_{i=2}^4\omega_{ii}^2\big)^2-\sum_{i=2}^4(\omega_{ii}^1)^4\\
=2\big\{(\omega_{22}^1\omega_{33}^1)^2+(\omega_{22}^1\omega_{44}^1)^2+(\omega_{33}^1\omega_{44}^1)^2\big\}.
\end{equation}
Combining \eqref{P1} with \eqref{P2} gives
\begin{align}\label{P3}
&\quad (f_1^2-f_2)^2-2(f_2^2-f_4)\\
&=4(\omega_{22}^1\omega_{33}^1+\omega_{22}^1\omega_{44}^1+\omega_{33}^1\omega_{44}^1)^2\nonumber\\
&\quad -4\big\{(\omega_{22}^1\omega_{33}^1)^2+(\omega_{22}^1\omega_{44}^1)^2+(\omega_{33}^1\omega_{44}^1)^2\big\}\nonumber\\
&=8\big\{(\omega_{22}^1)^2\omega_{33}^1\omega_{44}^1+\omega_{22}^1(\omega_{33}^1)^2\omega_{44}^1+\omega_{22}^1\omega_{33}^1(\omega_{44}^1)^2\big\}\nonumber\\
&=8f_1\omega_{22}^1\omega_{33}^1\omega_{44}^1.\nonumber
\end{align}
Similarly, we have
\begin{align}\label{P4}
f_1^3-f_3&=(\omega_{22}^1+\omega_{33}^1+\omega_{44}^1)^3
 -\big\{(\omega_{22}^1)^3+(\omega_{33}^1)^3+(\omega_{44}^1)^3\big\}\\
&=3\big\{(\omega_{22}^1)^2\omega_{33}^1+(\omega_{22}^1)^2\omega_{44}^1+(\omega_{33}^1)^2\omega_{22}^1+(\omega_{33}^1)^2\omega_{44}^1\nonumber\\
&\quad +(\omega_{44}^1)^2\omega_{22}^1+(\omega_{44}^1)^2\omega_{33}^1\big\}+6\omega_{22}^1\omega_{33}^1\omega_{44}^1\nonumber\\
&=3\sum_{i=2}^4(\omega_{ii}^1)^2(f_1-\omega_{ii}^1)+6\omega_{22}^1\omega_{33}^1\omega_{44}^1\nonumber\\
&=3f_1f_2-3f_3+6\omega_{22}^1\omega_{33}^1\omega_{44}^1\nonumber.
\end{align}
Eliminating $\omega_{22}^1\omega_{33}^1\omega_{44}^1$ from
\eqref{P3} and \eqref{P4}, we get \eqref{L4}.

 A direct computation
shows that
\begin{align*}\label{P6}
    f_1f_4&=\Big(\sum_{i=2}^4\omega_{ii}^1\Big)\Big(\sum_{i=2}^4(\omega_{ii}^1)^4\Big)\\
    &=\sum_{i=2}^4(\omega_{ii}^1)^5+\omega_{22}^1\Big\{(\omega_{33}^1)^4+(\omega_{44}^1)^4\Big\}\\
    &\quad +\omega_{33}^1\Big\{(\omega_{22}^1)^4+(\omega_{44}^1)^4\Big\}+\omega_{44}^1\Big\{(\omega_{22}^1)^4+(\omega_{33}^1)^4\Big\}\\
    &=f_5+\omega_{22}^1\left\{\big[(\omega_{33}^1)^2+(\omega_{44}^1)^2\big]^2-2(\omega_{33}^1)^2(\omega_{44}^1)^2\right\}\\
    &\quad +\omega_{33}^1\left\{\big[(\omega_{22}^1)^2+(\omega_{44}^1)^2\big]^2-2(\omega_{22}^1)^2(\omega_{44}^1)^2\right\}\\
    &\quad +\omega_{44}^1\left\{\big[(\omega_{22}^1)^2+(\omega_{33}^1)^2\big]^2-2(\omega_{22}^1)^2(\omega_{33}^1)^2\right\}\\
    &=f_5+\sum_{i=2}^4\omega_{ii}^1\Big(f_2-(\omega_{ii}^1)^2\Big)^2
    -2\omega_{22}^1\omega_{33}^1\omega_{44}^1\big(\omega_{22}^1\omega_{33}^1+\omega_{22}^1\omega_{44}^1+
    \omega_{33}^1\omega_{44}^1\big),
\end{align*}
which together with \eqref{P1} yields
\begin{equation}\label{P6}
    f_1f_4=2f_5+f_1f_2^2-2f_2f_3-(f_1^2-f_2)\omega_{22}^1\omega_{33}^1\omega_{44}^1.
\end{equation}
Eliminating $\omega_{22}^1\omega_{33}^1\omega_{44}^1$ from
\eqref{P4} and \eqref{P6} again, one gets
\begin{align}\label{P7}
6f_1f_4=12f_5+6f_1f_2^2-12f_2f_3-(f_1^2-f_2)(f_1^3-3f_1f_2+2f_3).
\end{align}
Moreover, eliminating the terms of $f_4$ from \eqref{L4} and
\eqref{P7} gives \eqref{L5}.
\end{proof}

For simplicity, we write  $\lambda=\lambda_1(t)$, $f_1=T$,
$T'=e_1(T)$, $T''=e_1e_1(T)$, $T'''=e_1e_1e_1(T)$ and
$T''''=e_1e_1e_1e_1(T)$. Then  the functions $f_1, \cdots, f_5$ are
expressed in the terms of $\lambda$ and $T$ in the following:
\begin{lemma} \label{lemma3.2}
 $f_1,~f_2,~f_3,~f_4,~~\mathrm {and}~f_5$ can be
written as
\begin{align}\label{L6}
\begin{cases}
f_1=T,\\
f_2=T'+3\lambda^2,\\
f_3=\frac{1}{2}T''-\lambda^2T+6\lambda\lambda',\\
f_4=\frac{1}{6}T'''-\frac{4}{3}\lambda^2T'-\frac{5}{3}\lambda\lambda'T+2\lambda'^2+4\lambda\lambda''-2\lambda^4,\\
f_5=\frac{1}{24}T''''-\frac{5}{6}\lambda^2T''-\frac{25}{12}\lambda\lambda'T'-
    \frac{1}{12}(13\lambda\lambda''+\lambda'^2-12\lambda^4)T\\
    \quad\quad+2\lambda\lambda'''+\frac{5}{3}\lambda'\lambda''-\frac{26}{3}\lambda^3\lambda'.
\end{cases}
\end{align}
\end{lemma}
\begin{proof}
Since $e_1(\lambda)\neq0$, $\lambda$ is not constant. From
\eqref{L1}, one has
\begin{align}\label{Q1}
f_1=\frac{e_1e_1(\lambda)-\lambda
S}{e_1(\lambda)}=\frac{\lambda''}{\lambda'}-\frac{\lambda}{\lambda'}S=:T.
\end{align}
Taking the sum of $i$ from 2 to 4 in \eqref{L3}-\eqref{L2}
respectively ang using \eqref{sum1}, we have
\begin{align}
f_2=&3\lambda^2+e_1(f_1)=T'+3\lambda^2,\label{Q2}\\
g_1:=&\sum_{i=2}^4 \lambda_i\omega_{ii}^1\label{Q3}\\
=&\lambda T-3e_1(\lambda)=\lambda T-3\lambda'.\nonumber
\end{align}
Multiplying by $\omega_{ii}^1$ both sides of equation \eqref{L3}, we
have
\begin{align*}
\frac{1}{2}e_1\big((\omega_{ii}^1)^2\big)&=(\omega_{ii}^1)^3+\lambda
\lambda_i\omega_{ii}^1.
\end{align*}
Taking the sum of $i$ in the above equation gives
\begin{align}
f_3=\frac 12 e_1(f_2)-\lambda  g_1\label{Q4}=\frac{1}{2}T''-\lambda^2T+6\lambda\lambda'.
\end{align}
Differentiating \eqref{Q3} along $e_1$, using \eqref{L2} and
\eqref{L3} we have
\begin{eqnarray}\label{Q5}
e_1(g_1)=2\sum_{i=2}^4\lambda_i\big(\omega_{ii}^1\big)^2+ \lambda
\sum_{i=2}^4\lambda_i^2-\lambda
\sum_{i=2}^4\big(\omega_{ii}^1\big)^2.
\end{eqnarray}
Hence, it follows from \eqref{sum1}, \eqref{sum2} and \eqref{Q1}
that
\begin{align}
g_2:=\sum_{i=2}^4\lambda_i\big(\omega_{ii}^1\big)^2
=&\frac{1}{2}\big\{e_1(g_1)-
\lambda \big(S-\lambda ^2\big)+\lambda f_2\big\}\nonumber\\
=&\frac{1}{2}\big\{e_1(g_1)-\lambda''+\lambda'T+\lambda^3+\lambda
f_2\big\}\nonumber.
\end{align}
Using \eqref{Q2} and \eqref{Q3}, the above expression reduces to
\begin{align}\label{Q6}
g_2=\lambda T'+\lambda'T-2\lambda''+2\lambda^3.
\end{align}
Multiplying $(\omega_{ii}^1)^2$ on both sides of equation
\eqref{L3}, we have
\begin{align*}
\frac{1}{3}e_1\big((\omega_{ii}^1)^3\big)&=(\omega_{ii}^1)^4+\lambda
\lambda_i(\omega_{ii}^1)^2.
\end{align*}
Taking the sum of  $i$ from 2 to 4 in the above equation, we obtain
\begin{align}\label{Q7}
f_4=&\frac{1}{3}e_1(f_3)-\lambda  g_2\\
=&\frac{1}{6}T'''-\frac{4}{3}\lambda^2T'-\frac{5}{3}\lambda\lambda'T+2\lambda'^2+4\lambda\lambda''-2\lambda^4.\nonumber
\end{align}
 Multiplying equation \eqref{L2} by $\lambda_i$
gives
\begin{align*}
\lambda_i^2\omega_{ii}^1=\frac{1}{2}e_1(\lambda_i^2)+\lambda
\lambda_i\omega_{ii}^1,
\end{align*}
which together with \eqref{sum2} yields
\begin{align}\label{Q8}
    g_3:&=\sum_{i=2}^4\lambda_i^2\omega_{ii}^1=\frac 12e_1(S-\lambda ^2)+\lambda g_1\\
    &=\frac 12\Big(\frac{\lambda''-\lambda'T}{\lambda}-\lambda ^2\Big)'+\lambda g_1\nonumber\\
    &=-\frac{\lambda'}{2\lambda} T'+\Big(\lambda^2-\frac{\lambda''\lambda-\lambda'^2}{2\lambda^2}\Big)T-
    4\lambda\lambda'+\frac{\lambda'''\lambda-\lambda''\lambda'}{2\lambda^2}\nonumber.
\end{align}
Differentiating \eqref{Q6} with respect to $e_1$ and using
\eqref{L2} and \eqref{L3}, we have
\begin{align}
e_1(g_2)=3\sum_{i=2}^4\lambda_i\big(\omega_{ii}^1\big)^3-\lambda
\sum_{i=2}^4\big(\omega_{ii}^1\big)^3+2\lambda
\sum_{i=2}^4\lambda_i^2\omega_{ii}^1,\nonumber
\end{align}
which leads to
\begin{align}\label{Q10}
    g_4:=&\sum_{i=2}^4\lambda_i\big(\omega_{ii}^1\big)^3\\
    =&\frac{1}{3}\big(e_1(g_2)+\lambda f_3-2\lambda g_3\big)\nonumber\\
    =&\frac 1 2 \lambda T''+\lambda'
    T'+\frac{1}{3}(2\lambda''-3\lambda^3-\frac{\lambda'^2}{\lambda})T\nonumber\\
    &-\lambda'''+\frac{20}{3}\lambda^2\lambda'+\frac{\lambda''\lambda'}{3\lambda}.\nonumber
\end{align}
Multiplying $(\omega_{ii}^1)^3$ on both sides of equation
\eqref{L3}, we have
\begin{align*}
\frac{1}{4}e_1\big((\omega_{ii}^1)^4\big)&=(\omega_{ii}^1)^5+\lambda
\lambda_i(\omega_{ii}^1)^3.
\end{align*}
After taking the sum of  $i$ in the above equation, we have
\begin{align}\label{Q11}
    f_5=&\frac{1}{4}e_1(f_4)-\lambda  g_4\\
    =&\frac{1}{24}T''''-\frac{5}{6}\lambda^2T''-\frac{25}{12}\lambda\lambda'T'-
    \frac{1}{12}(13\lambda\lambda''+\lambda'^2-12\lambda^4)T\nonumber\\
    &+2\lambda\lambda'''+\frac{5}{3}\lambda'\lambda''-\frac{26}{3}\lambda^3\lambda'.\nonumber
\end{align}
\end{proof}

\begin{lemma}\label{lemma3.3}
Let $M^4$ be an orientable biharmonic hypersurface with simple
distinct principal curvatures in $\mathbb R^5$.  Then the function
$T$ depends only on the variable $t$.
\end{lemma}
\begin{proof}
In the case $T=0$ in a region, it follows from \eqref{Q1} that $S$ depends only
on $t$. In the following, we assume that $T\neq0$ .

Substituting \eqref{L6} into \eqref{L4} and \eqref{L5} yields
\begin{align}
    &- T'''+ 4 T T''+ 3 T'^2+(-6 T^2 + 26 \lambda^2) T' +(T^4- 26 \lambda^2 T^2   \label{Equation1}\\
    &+ 58\lambda \lambda' T)+ 39 \lambda^4 - 24 \lambda \lambda'' - 12 \lambda'^2=0,\nonumber\\
   & -T''''  +10 T' T'' +(10 T^2  + 50 \lambda^2) T''- (20 T^3 +20 \lambda^2 T \label{Equation2} \\
   & - 170 \lambda \lambda') T' + (4 T^5  - 80\lambda^2 T^3   + 120\lambda \lambda' T^2
    - 84\lambda^4 T  +26 \lambda \lambda'' T \nonumber\\
   &+2\lambda'^2 T)    + 568\lambda^3 \lambda' - 48\lambda \lambda''' -40\lambda' \lambda''=0.\nonumber
\end{align}
Differentiating \eqref{Equation1} with respect to $e_1$, we have
\begin{align}
   & -T'''' +4TT''' +10 T' T'' +(-6T^2  + 26 \lambda^2) T''-12TT'^2\label{Equation3} \\
   &+ (4 T^3 -52 \lambda^2 T+110 \lambda \lambda') T'+ (-52\lambda \lambda' T^2\nonumber\\
   &+58\lambda \lambda'' T+58\lambda'^2T)+156\lambda^3\lambda'-48\lambda'\lambda'' - 24\lambda \lambda'''=0.\nonumber
\end{align}
Eliminating the terms on $T''''$ in \eqref{Equation2}-\eqref{Equation3}, we get
\begin{align}
   & 4TT''' -(16T^2 + 24 \lambda^2) T''-12TT'^2+ (24 T^3 -32 \lambda^2 T-60 \lambda \lambda') T'\label{Equation4}
   \\ &+(-4T^5+80\lambda^2T^3-172\lambda \lambda' T^2+84\lambda^4T+32\lambda \lambda'' T+56\lambda'^2T)\nonumber\\
   &-412\lambda^3\lambda'-8\lambda'\lambda'' +24\lambda \lambda'''=0.\nonumber
\end{align}
Moreover, we can eliminate the terms on $T'''$ of
\eqref{Equation1} and \eqref{Equation4}. Then we obtain
\begin{align}
   &-6\lambda^2T''+ (18\lambda^2 T-15 \lambda \lambda') T'\label{Equation5}
   \\ &+(-6\lambda^2T^3+15\lambda \lambda' T^2+60\lambda^4T-16\lambda \lambda'' T+2\lambda'^2T)\nonumber\\
   &-103\lambda^3\lambda'-2\lambda'\lambda'' +6\lambda \lambda'''=0.\nonumber
\end{align}
Differentiating the above equation along $e_1$, one sees
\begin{align}
   &-6\lambda^2T'''+ (18\lambda^2 T-27 \lambda \lambda') T''+18\lambda^2T'^2\label{Equation6}
   \\ &+(-18\lambda^2T^2+66\lambda \lambda' T+60\lambda^4-13\lambda'^2-31\lambda \lambda'')T'\nonumber\\
   &+(-12\lambda\lambda'T^3+15\lambda'^2T^2+15\lambda\lambda''T^2+240\lambda^3\lambda'T\nonumber\\
   &-12\lambda'\lambda''T-16\lambda\lambda'''T)-309\lambda^2\lambda'^2-103\lambda^3\lambda''\nonumber\\
   &-2\lambda''^2+4\lambda'\lambda''' +6\lambda
\lambda''''=0.\nonumber
\end{align}
Note that both equations \eqref{Equation1} and \eqref{Equation4}
have a non-zero term of $T^4$, but \eqref{Equation6} does not
involve any term of $T^4$. Therefore, we conclude that
\eqref{Equation6} are entirely different from equations
\eqref{Equation1} and \eqref{Equation4}, which are third-order
differential equations with respect to $T$.

Next, we   eliminate the
terms of $T'''$, $T''$, $T'$  and derive a non-trivial
equation of $T$.

 \noindent To eliminate the terms of $T'''$ from
\eqref{Equation1} and \eqref{Equation6}, we have
\begin{align}
   &(6\lambda^2 T+27 \lambda \lambda') T''+(-18\lambda^2T^2-66\lambda \lambda' T\label{Equation7}\\
   &+96\lambda^4+13\lambda'^2+31\lambda \lambda'')T'+(6\lambda^2T^4+12\lambda\lambda'T^3\nonumber\\
   &-156\lambda^4T^2-15\lambda'^2T^2-15\lambda\lambda''T^2+108\lambda^3\lambda'T\nonumber\\
   &+12\lambda'\lambda''T+16\lambda\lambda'''T)+234\lambda^6+237\lambda^2\lambda'^2-41\lambda^3\lambda''\nonumber\\
   &+2\lambda''^2-4\lambda'\lambda''' -6\lambda\lambda''''=0.\nonumber
\end{align}
Note that equations \eqref{Equation7} and \eqref{Equation5} are
entirely different. Then, multiplying $2\lambda T+9\lambda'$ to \eqref{Equation5}  and
$2\lambda$ to
\eqref{Equation7}, we eliminate the terms of $T''$ to
obtain
 \begin{align}
a_1T'-a_1T^2+a_2T+a_3=0,\label{Equation8}
\end{align}
where
\begin{align*}
a_1=&62\lambda^2 \lambda''-109\lambda \lambda'^2+192\lambda^5,\\
a_2=&44\lambda^2 \lambda'''- 124\lambda \lambda'
\lambda''+550\lambda^4 \lambda'+18\lambda'^3,\\
a_3=&- 12 \lambda^2 \lambda'''' + 46 \lambda \lambda' \lambda''' - 82 \lambda^4 \lambda''+ 4 \lambda \lambda''^2   \nonumber\\
      &- 18 \lambda'^2 \lambda''- 453 \lambda^3 \lambda'^2+468 \lambda^7.
\end{align*}
If $a_1=0$ in a region, then  \eqref{Equation8} becomes an equation of $T$
\begin{equation}
a_2T+a_3=0.\nonumber
\end{equation}
If $a_2\neq0$, we have $a_3\neq0$ as well since $T\neq0$. Hence the
conclusion follows obviously. If $a_2=0$, then the above equation
yields that $a_3=0$. We will derive a contradiction. Let us consider
the equations  $a_1=0$ and $a_2=0$:
\begin{align}
&62\lambda \lambda''-109\lambda'^2+192\lambda^4=0,\label{D1}\\
&44\lambda^2 \lambda'''- 124\lambda \lambda' \lambda''+550\lambda^4
\lambda'+18\lambda'^3=0.\label{D2}
\end{align}
It is easy to see the two equations \eqref{D1} and \eqref{D2} are
entirely different. Differentiating \eqref{D1} gives
\begin{align}
31\lambda
\lambda'''-78\lambda'\lambda''+384\lambda^3\lambda'=0,\nonumber
\end{align}
which together with \eqref{D2} reduces to
\begin{align}
-206\lambda \lambda''+279\lambda'^2+77\lambda^4=0.\label{D3}
\end{align}
After eliminating the terms of $\lambda''$ between \eqref{D1} and
\eqref{D3}, one has
\begin{align}\label{R1}
-2578\lambda'^2+22163\lambda^4=0.
\end{align}
Differentiating the above equation leads to
\begin{align}\label{R2}
-1289\lambda''+22163\lambda^3=0.
\end{align}
Substituting \eqref{R1} and \eqref{R2} into \eqref{D1}, we find
\begin{align}\
1066545669\lambda^4=0,\nonumber
\end{align}
 which is a contradiction since
$\lambda$ is non-constant. Hence we have  $a_1\neq0$.

\noindent Differentiating \eqref{Equation8} along $e_1$ again, we
deduce
\begin{align}
      & \big(192\lambda^5  - 109\lambda \lambda'^2   + 62\lambda^2 \lambda''\big) T''
       - \big(384\lambda^5 T   + 124\lambda^2 \lambda'' T   \label{Eq9}\\
       &- 218\lambda \lambda'^2 T  -1510\lambda^4 \lambda'
       -106\lambda^2 \lambda'''   +218\lambda \lambda' \lambda'' + 91\lambda'^3\big) T'  \nonumber\\
    &-\big(960\lambda^4 \lambda' + 62\lambda^2 \lambda''' -94\lambda \lambda' \lambda''
     - 109\lambda'^3\big) T^2  + \big(550\lambda^4 \lambda''  \nonumber\\
     & + 2200\lambda^3 \lambda'^2   + 44\lambda^2 \lambda''''
     - 36\lambda \lambda' \lambda'''  - 124\lambda \lambda''^2  - 70\lambda'^2 \lambda'' \big)T  \nonumber\\
     &+ 3276 \lambda^6 \lambda' - 1359 \lambda^2 \lambda'^3 - 82 \lambda^4 \lambda''' - 1234 \lambda^3 \lambda' \lambda''
     - 12 \lambda^2 \lambda''''' \nonumber\\
     &+ 22 \lambda \lambda' \lambda''''
     + 54 \lambda \lambda'' \lambda''' + 28 \lambda'^2 \lambda''' - 32 \lambda' \lambda''^2=0.\nonumber
\end{align}
Multiplying $192\lambda^4 - 109\lambda'^2 + 62\lambda \lambda'' $
and $6\lambda$ on the both sides of equations \eqref{Equation5} and
\eqref{Eq9} respectively, we thus get
\begin{align}
    & \big(1152 T \lambda^6 + 372 T  \lambda^3 \lambda''
    - 654 T \lambda^2 \lambda'^2 + 6180 \lambda^5 \lambda' \label{Eq10}\\
    &+ 636 \lambda^3 \lambda'''  - 2238  \lambda^2 \lambda' \lambda''
      + 1089  \lambda \lambda'^3\big)T' -\big(1152 \lambda^6 \nonumber\\
     &+372\lambda^3 \lambda'' - 654\lambda^2 \lambda'^2\big)T^3
    - \big(2880 \lambda^5 \lambda' +372 \lambda^3 \lambda''' \nonumber\\
     &-1494  \lambda^2 \lambda' \lambda''+ 981 \lambda
     \lambda'^3\big)T^2
      + \big(11520 \lambda^8 + 3948 \lambda^5 \lambda''\nonumber\\
     & + 7044 \lambda^4 \lambda'^2 + 264 \lambda^3 \lambda''''  - 216 \lambda^2 \lambda' \lambda'''
      - 1736 \lambda^2 \lambda''^2 \nonumber\\
      &+ 1448  \lambda \lambda'^2 \lambda'' - 218 \lambda'^4\big)T
      - 120 \lambda^7 \lambda'  + 660 \lambda^5 \lambda'''\nonumber\\
      &   - 14174 \lambda^4 \lambda' \lambda''+3073 \lambda^3 \lambda'^3 - 72 \lambda^3 \lambda''''' + 132 \lambda^2 \lambda' \lambda'''' \nonumber\\
      &+ 696 \lambda^2 \lambda'' \lambda'''
      - 486 \lambda \lambda'^2 \lambda''' - 316 \lambda \lambda' \lambda''^2 + 218 \lambda'^3 \lambda''=0.\nonumber
\end{align}
Since \eqref{Eq10} has a non-zero term of $T^3$, \eqref{Eq10} is
different from \eqref{Equation8}. Combining \eqref{Equation8} with
\eqref{Eq10}, we obtain a non-trivial polynomial equation concerning
$T$ with the coefficients depending only on the variable $t$
\begin{equation}\label{Eq11}
    b_1 T+b_2=0,
\end{equation}
where
\begin{align*}
    b_1=&209088  \lambda^{12} + 174078  \lambda^9 \lambda''
    - 309288  \lambda^8 \lambda'^2 + 8064  \lambda^7 \lambda'''' \\
    &- 89523  \lambda^6 \lambda' \lambda''' - 7830  \lambda^6 \lambda''^2
    + 302157  \lambda^5 \lambda'^2 \lambda'' - 227013  \lambda^4 \lambda'^4 \\
    &+ 2604  \lambda^4 \lambda'' \lambda'''' - 3498  \lambda^4 \lambda'''^2
    - 4578  \lambda^3 \lambda'^2 \lambda'''' + 18354  \lambda^3 \lambda' \lambda'' \lambda'''\\
    & - 13640  \lambda^3 \lambda''^3 - 717  \lambda^2 \lambda'^3 \lambda'''
    + 1350  \lambda^2 \lambda'^2 \lambda''^2 - 975  \lambda \lambda'^4 \lambda'' + 520  \lambda'^6,\\
    b_2=&- 364410 \lambda^{11} \lambda' - 21366 \lambda^9 \lambda'''
    - 146838 \lambda^8 \lambda' \lambda'' + 361623 \lambda^7 \lambda'^3 \\
    &- 1728 \lambda^7 \lambda''''' + 12438 \lambda^6 \lambda' \lambda''''
    + 28338 \lambda^6 \lambda'' \lambda''' - 20178 \lambda^5 \lambda'^2 \lambda'''\\
    & - 143462 \lambda^5 \lambda' \lambda''^2 + 120509 \lambda^4 \lambda'^3 \lambda''
     - 558 \lambda^4 \lambda'' \lambda''''' + 954 \lambda^4 \lambda''' \lambda''''\\
    & + 19795 \lambda^3 \lambda'^5 + 981 \lambda^3 \lambda'^2 \lambda'''''
    - 2334 \lambda^3 \lambda' \lambda'' \lambda'''' - 3657 \lambda^3 \lambda' \lambda'''^2 \\
    &+ 5076 \lambda^3 \lambda''^2 \lambda''' - 165 \lambda^2 \lambda'^3 \lambda''''
     + 1050 \lambda^2 \lambda'^2 \lambda'' \lambda''' - 1330 \lambda^2 \lambda' \lambda''^3\\
    & + 360 \lambda \lambda'^4 \lambda''' + 415 \lambda \lambda'^3 \lambda''^2 - 520 \lambda'^5 \lambda''.
\end{align*}
If $b_1=0$ and $b_2=0$ in a region, similarly, after eliminating the terms of
$\lambda'''''$, $\lambda''''$, $\lambda'''$, $\lambda''$, $\lambda'$
item by item, one gets a non-trivial polynomial equation of
$\lambda$ with constant coefficients, which implies that $\lambda$
is a constant. This is a contradiction.
Therefore, we conclude that $T$ depends only on the variable $t$.
\end{proof}

\begin{lemma}\label{lemma3.4}
Let $M^4$ be an orientable biharmonic hypersurface with non-constant
mean curvature in $\mathbb R^5$. Then $e_i(\lambda_j)=0$ for $2\leq
i, j\leq 4$, that is, all principal curvatures $\lambda_i$ depend
only on the variable $t$.
\end{lemma}
\begin{proof}
 If  the number $m$ of distinct principal curvatures
is 3 or 2, it has been proved in \cite{fu1}, so we only need  to
consider the case of four distinct principal curvatures. Since
$\lambda_i\neq \lambda_j$ at any point in $U_p$, it follows from
\eqref{L2} and \eqref{L3} that there exists a neighborhood $V\subset
U_p$ such that $\omega_{ii}^1\neq \omega_{jj}^1$ for $i\neq j$.

According to Lemmas 3.3, \eqref{L6} implies that $f_k$ for $k=1,
\ldots, 5$ depend only on the variable $t$, that is, $e_i(f_k)=0$
for $2\leq i\leq 4$. Hence, differentiating both sides of equations
$f_k=\sum_{i=2}^4(\omega_{ii}^1)^k$ for $k=1, 2, 3$ with respect to
$e_i$ $(2\leq i\leq 4)$, we obtain
\begin{align}\label{Vandermonde}
\begin{cases}
e_i(\omega_{22}^1)+e_i(\omega_{33}^1)+e_i(\omega_{44}^1)=0,\\
\omega_{22}^1 e_i(\omega_{22}^1)+\omega_{33}^1 e_i(\omega_{33}^1)+\omega_{44}^1 e_i(\omega_{44}^1)=0,\\
(\omega_{22}^1)^2 e_i(\omega_{22}^1)+(\omega_{33}^1)^2 e_i(\omega_{33}^1)+(\omega_{44}^1)^2 e_i(\omega_{44}^1)=0.\\
\end{cases}
\end{align}
Since $\omega_{22}^1$, $\omega_{33}^1$, $\omega_{44}^1$ are mutually
different at $V$ and the determinant of the coefficient matrix of
\eqref{Vandermonde} is the Vandermonde determinant with order 3, it
follows that
\begin{eqnarray*}
\left | \begin{array}{ccccc} 1&1&1\\
\omega_{22}^1&\omega_{33}^1&\omega_{44}^1\\
(\omega_{22}^1)^2&(\omega_{33}^1)^2&(\omega_{44}^1)^2\
\end{array} \right |=(\omega_{44}^1-\omega_{33}^1)(\omega_{44}^1-\omega_{22}^1)(\omega_{33}^1-\omega_{22}^1)\neq 0.
\end{eqnarray*}
According to Cramer's rule in linear algebra, one gets
\[e_i(\omega_{22}^1)= e_i(\omega_{33}^1)=e_i(\omega_{44}^1)=0.\]
Furthermore, for $j=2, 3, 4$ by considering
\begin{align}
e_ie_1(\omega_{jj}^1)-e_1e_i(\omega_{jj}^1)=[e_i,
e_1](\omega_{jj}^1)=\sum_{l=2}^4(\omega_{i1}^l-\omega_{1i}^l)e_l(\omega_{jj}^1),\nonumber
\end{align}
we get
\begin{align}
e_ie_1(\omega_{jj}^1)=0.\nonumber
\end{align}
Differentiating \eqref{L3} with respect to $e_i$, taking into
account the above equation and  $e_i(\omega_{jj}^1)=0$, we derive
$$e_i(\lambda_j)=0$$
for any $1\leq j\leq 4$ and $2\leq i\leq 4$. Therefore, we complete
a proof of Lemma 3.4.
\end{proof}
\begin{remark}
Note that our method developed in Lemmas
\ref{lemma3.2}-\ref{lemma3.4} works also for $n=2$ and $n=3$, which
were proved with different approaches in \cite{jiang1987},
\cite{HasanisVlachos1995}, or \cite{defever1998}.
\end{remark}

\section{Proof of the main theorem}
We first recall some relations concerning the coefficients of
connection and principal curvature functions verified in
\cite{fu-hong2018} or \cite{fu3}. Actually, they are the four
dimension version of Lemmas 3.5 and 3.6 in \cite{fu-hong2018} while
Lemma \ref{lemma3.4} holds.
\begin{lemma}{\rm([20])}\label{lemma4.1}
For three distinct principal curvatures $\lambda_i$, $\lambda_j$ and
$\lambda_k$ $(2\leq i, j, k\leq 4)$, we have the following
relations:
\begin{align}
&\omega_{23}^4(\lambda_3-\lambda_4)=\omega_{32}^4(\lambda_2-\lambda_4)=\omega_{43}^2(\lambda_3-\lambda_2),\label{F1}\\
&\omega_{23}^4\omega_{32}^4+\omega_{34}^2\omega_{43}^2+\omega_{24}^3\omega_{42}^3=0,\label{F2}\\
&\omega_{23}^4(\omega_{33}^1-\omega_{44}^1)=\omega_{32}^4(\omega_{22}^1-\omega_{44}^1)=
\omega_{43}^2(\omega_{33}^1-\omega_{22}^1). \label{F3}
\end{align}
\end{lemma}
\begin{lemma} {\rm([20])}
Under the assumptions as above, we have
\begin{align}
\omega_{22}^1\omega_{33}^1-2\omega_{23}^4\omega_{32}^4=-\lambda_2\lambda_3,  \label{F4}\\
\omega_{22}^1\omega_{44}^1-2\omega_{24}^3\omega_{42}^3=-\lambda_2\lambda_4, \label{F5}\\
\omega_{33}^1\omega_{44}^1-2\omega_{34}^2\omega_{43}^2=-\lambda_3\lambda_4.\label{F6}
\end{align}
\end{lemma}

\begin{proof}[{\bf The proof of Theorem 1.1}] If the mean curvature $H$ is
constant, the first equation of \eqref{biharmonic condition} reduces
$H=0$ immediately. Assume that the mean curvature $H$ is
non-constant on $U_p$. According to Lemma \ref{lemma4.1}, we
distinguish it into the following cases A and B.
\medskip

\noindent
{\bf Case A}. $\omega_{23}^4\neq0$, $\omega_{32}^4\neq0$ and
$\omega_{43}^2\neq0$. 

In this case, equations \eqref{F1} and
\eqref{F3} reduce to
\begin{align*}
\frac{\omega_{33}^1-\omega_{44}^1}{\lambda_3-\lambda_4}
&=\frac{\omega_{33}^1-\omega_{22}^1}{\lambda_3-\lambda_2}
=\frac{\omega_{44}^1-\omega_{22}^1}{\lambda_4-\lambda_2},
\end{align*}
and hence there exist two smooth functions $\alpha$ and $\beta$
depending on $t$ such that
\begin{align}
\omega_{ii}^1=\alpha \lambda_i+\beta\label{F7}
\end{align}
for $i=2, 3, 4$. Differentiating with respect to $e_1$ on both sides
of equation \eqref{F7}, using \eqref{L2} and \eqref{L3} we get
\begin{align}
&e_1(\alpha)=\lambda_1(\alpha^2+1)+\alpha\beta,\label{F8}\\
&e_1(\beta)=\beta(\alpha\lambda_1+\beta).\label{F9}
\end{align}
Taking a sum on $i$ in \eqref{F7}  and using \eqref{sum1}, one has
\begin{align}
\sum_{i=2}^4\omega_{ii}^1=-3\alpha\lambda_1+3\beta.\label{F10}
\end{align}
Taking into account \eqref{F2}, equations \eqref{F4}, \eqref{F5} and
\eqref{F6} lead to
\begin{align*}
\omega_{22}^1\omega_{33}^1+\omega_{22}^1\omega_{44}^1+\omega_{33}^1\omega_{44}^1=
-\lambda_2\lambda_3-\lambda_2\lambda_4-\lambda_3\lambda_4,
\end{align*}
 which together with \eqref{F7} further reduces to
\begin{align}
(1+\alpha^2)(\lambda_2\lambda_3+\lambda_2\lambda_4+\lambda_3\lambda_4)+
2\alpha\beta(\lambda_2+\lambda_3+\lambda_4)+3\beta^2=0.\label{F11}
\end{align}
Since $S-\lambda_1^2=\sum_{i=2}^4\lambda_i^2$ and
$-3\lambda_1=\sum_{i=2}^4\lambda_i$, it follows from \eqref{F11}
that
\begin{align*}
(1+\alpha^2)\times \frac{1}{2}(10\lambda_1^2-S)-
6\alpha\beta\lambda_1+3\beta^2=0
\end{align*}
and hence
\begin{align}
(1+\alpha^2)S=10(1+\alpha^2)\lambda_1^2-
12\alpha\beta\lambda_1+6\beta^2.\label{F12}
\end{align}
Differentiating \eqref{F12} with respect to $e_1$ and using \eqref{F8}-\eqref{F9}, one has
\begin{align}
(1+\alpha^2)e_1(S)=&4\big\{5(1+\alpha^2)\lambda_1-3\alpha\beta\big\}e_1(\lambda_1)\label{F13}\\
&+2(10\alpha\lambda_1^2-\alpha
S-6\beta\lambda_1)\big\{\lambda_1(\alpha^2+1)+\alpha\beta\big\}\nonumber\\
&+12\beta(\beta-\alpha\lambda_1)(\alpha\lambda_1+\beta).\nonumber
\end{align}
Moreover, differentiating \eqref{sum1} with respect to $e_1$ and
using \eqref{L2}, we get
\begin{align}
-3e_1(\lambda_1)=\sum_{i=2}^4(\lambda_i-\lambda_1)\omega_{ii}^1.\nonumber
\end{align}
Using \eqref{F7} and \eqref{sum1},    the above equation
reduces to
\begin{align}
-3e_1(\lambda_1)&=\sum_{i=2}^4(\lambda_i-\lambda_1)(\alpha
\lambda_i+\beta) \label{F14}\\
&=\sum_{i=2}^4\big\{\alpha
\lambda_i^2+(\beta-\lambda_1\alpha)\lambda_i-\beta\lambda_1\big\}\nonumber\\
&=\alpha(S-\lambda_1^2)-3(\beta-\lambda_1\alpha)\lambda_1-3\beta\lambda_1 \nonumber\\
&=\alpha(2\lambda_1^2+S)-6\beta\lambda_1.\nonumber
\end{align}
By using \eqref{F12} and \eqref{F14}, we
eliminate $S$ to get
\begin{align}\label{F15}
   e_1(\lambda_1)=-\frac {1}{1+\alpha^2}(4\lambda_1^2 \alpha^3 -6\lambda_1\alpha^2\beta+2\alpha\beta^2+4\lambda_1^2\alpha-2\lambda_1\beta).
\end{align}
On the other hand, differentiating \eqref{F14} with respect to
$e_1$, it follows from \eqref{F8} and \eqref{F9} that
\begin{align}
-3e_1e_1(\lambda_1)=&(4\alpha\lambda_1-6\beta)e_1(\lambda_1)+\alpha
e_1(S)\label{F16}\\
&+(2\lambda_1^2+S)\big\{\lambda_1(\alpha^2+1)+\alpha\beta\big\}-6\beta\lambda_1(\alpha\lambda_1+\beta).\nonumber
\end{align}
Substituting \eqref{F10} into \eqref{L1} gives
\begin{align}
e_1e_1(\lambda_1)=3(-\alpha\lambda_1+\beta)e_1(\lambda_1)+\lambda_1S.\label{F17}
\end{align}
Eliminating the terms of $e_1e_1(\lambda_1)$ between \eqref{F16} and
\eqref{F17} yields
\begin{eqnarray}
  &(-5\alpha\lambda_1+3\beta)e_1(\lambda_1)+\alpha
e_1(S)+(2\lambda_1^2+S)\big\{\lambda_1(\alpha^2+1)+\alpha\beta\big\}\label{F18}\\
&-6\beta\lambda_1(\alpha\lambda_1+\beta)+3\lambda_1 S=0.\nonumber
\end{eqnarray}
Combining \eqref{F18} with \eqref{F14} we may eliminate
$e_1(\lambda_1)$ and hence
\begin{align}
3\alpha e_1(S)=&(-5\alpha\lambda_1+3\beta)\big\{(2\lambda_1^2+S)\alpha-6\beta\lambda_1\big\} \label{F19}\\
&-3(2\lambda_1^2+S)\big\{\lambda_1(\alpha^2+1)+\alpha\beta\big\}+18\beta\lambda_1(\alpha\lambda_1+\beta)-9\lambda_1
S.\nonumber
\end{align}
Also, combining \eqref{F13} with \eqref{F14}, we eliminate
$e_1(\lambda_1)$ to obtain
\begin{align}
3(1+\alpha^2)e_1(S)=&4\big\{5(1+\alpha^2)\lambda_1-3\alpha\beta\big\}\big\{6\beta\lambda_1-(2\lambda_1^2+S)\alpha\big\}\label{F20}\\
&+6(10\alpha\lambda_1^2-\alpha
S-6\beta\lambda_1)\big\{\lambda_1(\alpha^2+1)+\alpha\beta\big\}\nonumber\\
&+36\beta(\beta-\alpha\lambda_1)(\alpha\lambda_1+\beta).\nonumber
\end{align}
Eliminating the terms of $e_1(S)$ in \eqref{F19}-\eqref{F20} yields
\begin{align}
&\big\{(1+\alpha^2)(5\alpha\lambda_1+\beta)-4\alpha^2\beta\big\}\big\{6\beta\lambda_1-(2\lambda_1^2+S)\alpha\big\} \label{F21}\\
&+(22\alpha^2\lambda_1^2-12\alpha\beta\lambda_1+2\lambda_1^2+S-\alpha^2S)\big\{\lambda_1(\alpha^2+1)+\alpha\beta\big\}\nonumber\\
&+6\beta(2\alpha\beta-3\alpha^2\lambda_1-\lambda_1)(\alpha\lambda_1+\beta)
+3(1+\alpha^2)\lambda_1S=0.\nonumber
\end{align}
Applying \eqref{F12} to eliminate the terms of $S$ in \eqref{F21},
we derive
\begin{align}
&\big\{(1+\alpha^2)(5\alpha\lambda_1+\beta)-4\alpha^2\beta\big\}\big\{6\beta\lambda_1-2\alpha\lambda_1^2\big\} \nonumber\\
&+(22\alpha^2\lambda_1^2-12\alpha\beta\lambda_1+2\lambda_1^2)\big\{\lambda_1(\alpha^2+1)+\alpha\beta\big\}\nonumber\\
&+6\beta(2\alpha\beta-3\alpha^2\lambda_1-\lambda_1)(\alpha\lambda_1+\beta)\nonumber\\
&+\big\{-\alpha(5\alpha\lambda_1+\beta)+(1-\alpha^2)\lambda_1+3\lambda_1\big\}\big\{10(1+\alpha^2)\lambda_1^2-
12\alpha\beta\lambda_1+6\beta^2\big\}\nonumber\\
&+\frac{\alpha\beta(1+3\alpha^2)}{1+\alpha^2}\big\{10(1+\alpha^2)\lambda_1^2-
12\alpha\beta\lambda_1+6\beta^2\big\}=0\nonumber,
\end{align}
which reduces to
\begin{align}\label{F22}
    &8\alpha^6 \lambda_1^3-20\alpha^5\beta\lambda_1^2+16\alpha^4\beta^2\lambda_1-4\alpha^3\beta^3\\
    &+9\alpha^4 \lambda_1^3-14\alpha^3\beta\lambda_1^2+8\alpha^2\beta^2\lambda_1-2\alpha\beta^3\nonumber\\
    &-6\alpha^2 \lambda_1^3+6\alpha\beta\lambda_1^2-4\beta^2\lambda_1-7\lambda_1^3=0.\nonumber
\end{align}
Differentiating \eqref{F22} along $e_1$ and using \eqref{F8},
\eqref{F9} and \eqref{F15}, we have
\begin{align*}
    &\Big\{ 3\lambda_1^2(8\alpha^6+9\alpha^4-6\alpha^2-7)+2\lambda_1(6\alpha\beta-14\alpha^3\beta-20\alpha^5\beta)\\
    &+(16\alpha^4\beta^2+8\alpha^2\beta^2-4\beta^2) \Big\}
    \Big\{ 4\lambda_1^2 \alpha^3 -6\lambda_1\alpha^2\beta+2\alpha\beta^2+4\lambda_1^2\alpha-2\lambda_1\beta \Big\}\\
    &-(1+\alpha^2)\Big\{\lambda_1 (\alpha^2+1)+\alpha\beta\Big\}\Big\{48\lambda_1^3\alpha^5-100\lambda_1^2\alpha^4\beta
    +4\alpha^3(16\lambda_1 \beta^2+9\lambda_1^3)\\
    &-3\alpha^2(4\beta^3+14\lambda_1^2\beta)
    +2\alpha(8\lambda_1\beta^2-6\lambda_1^3)+(6\lambda_1^2\beta-2\beta^3)  \Big\}\\
    &-(1+\alpha^2)(\lambda_1 \alpha\beta+\beta^2)\Big\{-3\beta^2(4\alpha^3+2\alpha)
    +2\beta(16\lambda_1\alpha^4+8\lambda_1\alpha^2-4\lambda_1)\\
    & +(6\lambda_1^2\alpha-14\lambda_1^2\alpha^3-20\lambda_1^2\alpha^5)\Big\}=0,
\end{align*}
which leads to
\begin{align}\label{F23}
        &24\alpha^9\lambda_1^4 - 116\alpha^8\beta\lambda_1^3 + 188\alpha^7\beta^2\lambda_1^2 - 124\alpha^6\beta^3\lambda_1\\
        & + 36\alpha^7\lambda_1^4  - 145\alpha^6\beta\lambda_1^3 + 191\alpha^5\beta^2\lambda_1^2 - 110\alpha^4\beta^3\lambda_1\nonumber\\
        & - 36\alpha^5\lambda_1^4+76\alpha^4\beta\lambda_1^3 - 48\alpha^3\beta^2\lambda_1^2 + 15\alpha^2\beta^3\lambda_1\nonumber\\
        & - 84\alpha^3\lambda_1^4  + 123\alpha^2\beta\lambda_1^3- 51\alpha\beta^2\lambda_1^2 + 9\beta^3\lambda_1\nonumber\\
        & - 36\alpha\lambda_1^4  + 18\beta\lambda_1^3+ 28\alpha^5\beta^4 + 24\alpha^3\beta^4=0.\nonumber
\end{align}
To simplify the notations, \eqref{F22} and \eqref{F23}
can be respectively rewritten as
\begin{align}
&P_{13} \beta^3+P_{12} \beta^2 +P_{11}\beta+P_{10}=0,\label{F24}\\
&P_{24}\beta^4+P_{23} \beta^3+P_{22} \beta^2
+P_{21}\beta+P_{20}=0,\label{F25}
\end{align}
where
\begin{align*}
    &P_{13}=-4\alpha^3-2\alpha,\qquad
    P_{12}=16\alpha^4 \lambda_1+8\alpha^2 \lambda_1-4 \lambda_1,\\
    &P_{11}=-20\alpha^5 \lambda_1^2-14\alpha^3 \lambda_1^2+6\alpha \lambda_1^2,\qquad
    P_{10}=8\alpha^6 \lambda_1^3+9\alpha^4 \lambda_1^3-6\alpha^2 \lambda_1^3-7\lambda_1^3,\\
    &P_{24}=28\alpha^5 + 24\alpha^3,\qquad
    P_{23}= - 124\alpha^6\lambda_1- 110\alpha^4\lambda_1+ 15\alpha^2\lambda_1 + 9\lambda_1 ,\\
    &P_{22}=188\alpha^7 \lambda_1^2+ 191\alpha^5 \lambda_1^2- 48\alpha^3 \lambda_1^2 - 51\alpha \lambda_1^2  ,\\
    &P_{21}=- 116\alpha^8 \lambda_1^3- 145\alpha^6 \lambda_1^3+76\alpha^4 \lambda_1^3+ 123\alpha^2 \lambda_1^3+ 18 \lambda_1^3 ,\\
     &P_{20}=24\alpha^9\lambda_1^4+ 36\alpha^7\lambda_1^4- 36\alpha^5\lambda_1^4- 84\alpha^3\lambda_1^4- 36\alpha\lambda_1^4.
\end{align*}
Multiplying       \eqref{F24}  by $P_{24}\beta$ and \eqref{F25} by $P_{13}$,  we
eliminate the terms of $\beta^4$ to obtain
\begin{equation*}
(P_{12}P_{24}-P_{13}P_{23})\beta^3+(P_{11}P_{24}-P_{13}P_{22})\beta^2+(P_{10}P_{24}-P_{13}P_{21})\beta-P_{13}P_{20}=0,
\end{equation*}
which can be rewritten as
\begin{equation}\label{F26}
    P_{33}\beta^3+P_{32}\beta^2+P_{31}\beta+P_{30}=0,
\end{equation}
where
\begin{align*}
    &P_{33}=P_{12}P_{24}-P_{13}P_{23}=(-48\alpha^9-80\alpha^7-80\alpha^5-30\alpha^3+18\alpha)\lambda_1,\\
    &P_{32}=P_{11}P_{24}-P_{13}P_{22}=(192\alpha^{10}+ 268\alpha^8+22\alpha^6-156\alpha^4-102\alpha^2)\lambda_1^2,\\
    &P_{31}=P_{10}P_{24}-P_{13}P_{21}=(-240\alpha^{11}-368\alpha^9+62\alpha^7+304\alpha^5+150\alpha^3+36\alpha)\lambda_1^3,\\
    &P_{30}=-P_{13}P_{20}=(96\alpha^{12}+ 192\alpha^{10}-72\alpha^8-408\alpha^6-312\alpha^4-72\alpha^2) \lambda_1^4.
\end{align*}
Multiplying \eqref{F24} by $P_{33}$   and  \eqref{F26} by $P_{13}$,
 we  eliminate the terms on $\beta^3$ to  get
\begin{equation}\label{F27}
    P_{42}\beta^2+P_{41}\beta+P_{40}=0,
\end{equation}
where
\begin{align*}
    P_{42}&=P_{12}P_{33}-P_{13}P_{32}\\
    &=(-208\alpha^{11}-1104\alpha^9-1380\alpha^7-352\alpha^5+60\alpha^3-72\alpha)\lambda_1^2,\\
    P_{41}&=P_{11}P_{33}-P_{13}P_{31}\\
    &=(320\alpha^{12}+ 1944\alpha^{10}+2580\alpha^8+788\alpha^6+12\alpha^4+180\alpha^2)\lambda_1^3,\\
    P_{40}&=P_{10}P_{33}-P_{13}P_{30}\\
    &=(-112\alpha^{13}-976\alpha^{11}-1920\alpha^9-1150\alpha^7-10\alpha^5-42\alpha^3-126\alpha)\lambda_1^4.
\end{align*}
Multiplying \eqref{F27} by $P_{13}\beta$ and using \eqref{F24} again, one has
\begin{align}\label{F28}
    P_{52}\beta^2+P_{51}\beta+P_{50}=0,
\end{align}
where
\begin{align*}
    P_{52}=&(-2048\alpha^{15}-10912\alpha^{13}-15872\alpha^{11}-3944\alpha^9\\
    & +5288\alpha^7+1480\alpha^5-456\alpha^3+288\alpha)\lambda_1^3,\\
    P_{51}=&(3712\alpha^{16}+20864\alpha^{14}+32176\alpha^{12}+ 11296\alpha^{10}\\
    &-6892\alpha^8-1700\alpha^6+780\alpha^4-684\alpha^2)\lambda_1^4,\\
    P_{50}=&(-1664\alpha^{17}-10704\alpha^{15}-19728\alpha^{13}-7156\alpha^{11}\\
    & +13320\alpha^9+11736\alpha^7+1456\alpha^5+12\alpha^3+504\alpha)\lambda_1^5.
\end{align*}
Eliminating the terms of $\beta^2$ in \eqref{F27}-\eqref{F28} gives
\begin{align}\label{F29}
    (P_{41}P_{52}-P_{42}P_{51})\beta=-(P_{40}P_{52}-P_{42}P_{50}),
\end{align}
where
\begin{align*}
    P_{41}P_{52}-P_{42}P_{51}=&\Big\{116736\alpha^{27}+964608\alpha^{25}+3273216\alpha^{23}\\
    &+6086784\alpha^{21}+7013568\alpha^{19}+5538720\alpha^{17} \\
    & +3434064\alpha^{15}+1872528\alpha^{13}+852864\alpha^{11}+289728\alpha^9\\
    &+77904\alpha^7+18576\alpha^5+2592\alpha^3\Big\}\lambda_1^6,\\
    -(P_{40}P_{52}-P_{42}P_{50})=&\Big\{116736\lambda^7\alpha^{28} + 842496\lambda^7\alpha^{26} + 1857024\lambda^7\alpha^{24} \\
    &- 613632\lambda^7\alpha^{22} - 10278336\lambda^7\alpha^{20} - 21542832\lambda^7\alpha^{18} \\
    &- 23399376\lambda^7\alpha^{16} - 15074160\lambda^7\alpha^{14} - 5949072\lambda^7\alpha^{12} \\
    &- 1571856\lambda^7\alpha^{10} - 402288\lambda^7\alpha^8 - 111312\lambda^7\alpha^6 - 15984\lambda^7\alpha^4\Big\}\lambda_1^7.
\end{align*}
By substituting \eqref{F29} into \eqref{F27}, one  derives
\begin{align}\label{F30}
    &P_{42}(P_{40}P_{52}-P_{42}P_{50})^2-P_{41}(P_{40}P_{52}-P_{42}P_{50})(P_{41}P_{52}-P_{42}P_{51})\\
    &+P_{40}(P_{41}P_{52}-P_{42}P_{51})^2=0,\nonumber
\end{align}
namely,
\begin{align*}
    &\Big\{484844765184\alpha^{65} + 12599707041792\alpha^{63} + 168848312500224\alpha^{61} \\
&+ 1539102781734912\alpha^{59} + 10358767180578816\alpha^{57} + 52901614335688704\alpha^{55} \\
&+ 208041234970705920\alpha^{53} + 638282602571268096\alpha^{51} + 1546611187674415104\alpha^{49}\\
& + 2991538078614835200\alpha^{47} + 4659369092304187392\alpha^{45} + 5883550040047592448\alpha^{43}\\
& + 6057015443197185024\alpha^{41} + 5112599987173082112\alpha^{39} + 3566610599450135040\alpha^{37}\\
& + 2084666842179740160\alpha^{35} + 1044195650378735616\alpha^{33} + 461769373555605504\alpha^{31}\\
& + 184830143553160704\alpha^{29} + 67312785426974208\alpha^{27} + 22101468989783040\alpha^{25} \\
&+ 6548262801067008\alpha^{23} + 1779181049894400\alpha^{21} + 441651900169728\alpha^{19} \\
&+ 96378291542016\alpha^{17} + 18243537610752\alpha^{15} + 3112242227712\alpha^{13}\\
& + 445229250048\alpha^{11} + 38268370944\alpha^9 + 846526464\alpha^7\Big\}\lambda_1^{16}=0.
\end{align*}
Since $\lambda_1\neq 0$, the above equation is a non-trivial polynomial equation concerning $\alpha$ with constant coefficients. This implies that $\alpha$ must be a constant.
It follows from \eqref{F8} that
\begin{equation}\label{F31}
    \beta=-\frac {\alpha^2+1}{\alpha}\lambda_1.
\end{equation}
Substituting \eqref{F31} into \eqref{F9} and \eqref{F15},   we have
\begin{align}
    e_1(\lambda_1)&=\lambda_1\Big(\alpha\lambda_1-\frac{1+\alpha^2}{\alpha}\lambda_1\Big)=-\frac {1}{\alpha}\lambda_1^2,\label{F32}\\
    e_1(\lambda_1)&=-\frac{4\alpha^3+6\alpha(\alpha^2+1)+\frac{2(1+\alpha^2)^2}{\alpha}+4\alpha+\frac{2(1+\alpha^2)}{\alpha}}{1+\alpha^2}\lambda_1^2\label{F33}\\
    &=-\frac{4(3\alpha^2+1)}{\alpha}\lambda_1^2.\nonumber
\end{align}
Combining \eqref{F32} with \eqref{F33} yields that $4\alpha^2+1=0$, which is impossible.
\medskip

\noindent
{\bf Case B.}   $\omega_{23}^4=\omega_{32}^4=\omega_{43}^2=0$ at any
point $p$ in $V$. 

In this case, it follows  from \eqref{F4},
\eqref{F5}, \eqref{F6} that
\begin{align}
\omega_{22}^1\omega_{33}^1=-\lambda_2\lambda_3,  \label{G1}\\
\omega_{22}^1\omega_{44}^1=-\lambda_2\lambda_4, \label{G2}\\
\omega_{33}^1\omega_{44}^1=-\lambda_3\lambda_4.\label{G3}
\end{align}
Then we divide Case B into the two subcases B.1 and B.2.

\noindent
{\bf Case B.1.}  All principal curvatures $\lambda_2$,
$\lambda_3$, $\lambda_4$ are nonzero.

It follows from \eqref{G1}, \eqref{G2}, \eqref{G3} that
all $\omega_{22}^1$, $\omega_{33}^1$, $\omega_{44}^1$ are nonzero as well. Combining \eqref{G1} with
\eqref{G2} gives
\begin{align}
\frac{\omega_{33}^1}{\omega_{44}^1}=\frac{\lambda_3}{\lambda_4},\nonumber
\end{align}
which together with \eqref{G3} gives
\begin{align}
(\omega_{33}^1)^2+\lambda_3^2=0.\nonumber
\end{align}
It follows that  $\lambda_3=\omega_{33}^1=0$, which is a contradiction.

\noindent
{\bf Case B.2.} At least one of principal curvatures $\lambda_2$,
$\lambda_3$, $\lambda_4$ is zero. 

Without loss of generality, we
assume $\lambda_4=0$ at some point $p$ on $M^4$. It follows from
\eqref{L2} that $\omega_{44}^1=0$ at $p$. Hence \eqref{sum1} becomes
\begin{eqnarray}\label{G4}
\lambda_2+\lambda_3=-3\lambda_1.
\end{eqnarray}
Differentiating \eqref{G4} with respect to $e_1$ and
using \eqref{L2}, we have
\begin{align}\label{G5}
-3e_1(\lambda_1)=(\lambda_2-\lambda_1)\omega_{22}^1+(\lambda_3-\lambda_1)\omega_{33}^1.
\end{align}
Differentiating \eqref{G5} with respect to $e_1$, we apply
\eqref{L2}-\eqref{L3} to get
\begin{align}\label{G6}
-3e_1e_1(\lambda_1)=&2(\lambda_2-\lambda_1)(\omega_{22}^1)^2
+2(\lambda_3-\lambda_1)(\omega_{33}^1)^2\\
&-e_1(\lambda_1)(\omega_{22}^1+\omega_{33}^1)+(\lambda_2-\lambda_1)\lambda_2 \lambda_1+
(\lambda_3-\lambda_1)\lambda_3 \lambda_1.\nonumber
\end{align}
Furthermore, \eqref{L1} becomes
\begin{align}\label{G7}
e_1e_1(\lambda_1)=e_1(\lambda_1)(\omega_{22}^1+\omega_{33}^1)+ \lambda_1
(\lambda_1^2+\lambda_2^2+\lambda_3^2).
\end{align}
Eliminating the terms of $e_1e_1(\lambda_1)$ between \eqref{G6} and
\eqref{G7}, it gives
\begin{align}
&2e_1(\lambda_1)(\omega_{22}^1+\omega_{33}^1)+2(\lambda_2-\lambda_1)(\omega_{22}^1)^2
+2(\lambda_3-\lambda_1)(\omega_{33}^1)^2\label{G8}\\
&+(\lambda_2-\lambda_1)\lambda_2 \lambda_1+
(\lambda_3-\lambda_1)\lambda_3 \lambda_1+3\lambda_1
(\lambda_1^2+\lambda_2^2+\lambda_3^2)=0.\nonumber
\end{align}
Multiplying $\omega_{22}^1$ or $\omega_{33}^1$ to the both sides of equation \eqref{G5} respectively, it follows from  using \eqref{G1} that
\begin{align*}
(\lambda_2-\lambda_1)(\omega_{22}^1)^2=-3e_1(\lambda_1)\omega_{22}^1+(\lambda_3-\lambda_1)\lambda_2\lambda_3,\\
(\lambda_3-\lambda_1)(\omega_{33}^1)^2=-3e_1(\lambda_1)\omega_{33}^1+(\lambda_2-\lambda_1)\lambda_2\lambda_3.
\end{align*}
Substituting the above two equation into \eqref{G8}, we obtain
\begin{align}
&-4e_1(\lambda_1)(\omega_{22}^1+\omega_{33}^1)+2(\lambda_3-\lambda_1)\lambda_2\lambda_3
+2(\lambda_2-\lambda_1)\lambda_2\lambda_3\label{G9}\\
&+(\lambda_2-\lambda_1)\lambda_2 \lambda_1+
(\lambda_3-\lambda_1)\lambda_3 \lambda_1+3\lambda_1
(\lambda_1^2+\lambda_2^2+\lambda_3^2)=0.\nonumber
\end{align}
Taking into account \eqref{G4} and eliminating $\lambda_3$ in \eqref{G9}, one has
\begin{align}\label{G10}
2e_1(\lambda_1)(\omega_{22}^1+\omega_{33}^1)=9\lambda_1\lambda_2^2+27\lambda_1^2\lambda_2+21\lambda_1^3.
\end{align}
It follows from \eqref{G5}, \eqref{G1} and \eqref{G4} that \eqref{G10} is simplified to
\begin{align}\label{G11}
2(\lambda_2-\lambda_1)(\omega_{22}^1)^2+2(\lambda_3-\lambda_1)(\omega_{33}^1)^2
=-17\lambda_1\lambda_2^2-51\lambda_1^2\lambda_2-63\lambda_1^3.
\end{align}
Differentiating \eqref{G10} with respect to $e_1$, it follows from using \eqref{F2} and \eqref{F3} that
\begin{align*}
&2e_1e_1(\lambda_1)(\omega_{22}^1+\omega_{33}^1)+2e_1(\lambda_1)\Big\{(\omega_{22}^1)^2+(\omega_{33}^1)^2+
\lambda_1\lambda_2+\lambda_1\lambda_3\Big\}\\
&=(9\lambda_2^2+54\lambda_1\lambda_2+63\lambda_1^2)e_1(\lambda_1)+(18\lambda_1\lambda_2+27\lambda_1^2)(\lambda_2-\lambda_1)
\omega_{22}^1,
\end{align*}
which, together with \eqref {G7}, \eqref{G4} and \eqref{G1},  implies
\begin{align}\label{G12}
&e_1(\lambda_1)\Big\{4(\omega_{22}^1)^2+4(\omega_{33}^1)^2-69\lambda_1^2-42\lambda_1\lambda_2-5\lambda_2^2\Big\}\\
&+(47\lambda_1^3+3\lambda_1^2\lambda_2-14\lambda_1\lambda_2^2)\omega_{22}^1+(20\lambda_1^3+12\lambda_1^2\lambda_2+4\lambda_1\lambda_2^2)\omega_{33}^1=0.\nonumber
\end{align}
Eliminating the terms of $e_1(\lambda_1)$ between \eqref{G10} and
\eqref{G12} yields
\begin{align}
&(9\lambda_1\lambda_2^2+27\lambda_1^2\lambda_2+21\lambda_1^3)\Big\{4(\omega_{22}^1)^2+
4(\omega_{33}^1)^2-69\lambda_1^2-42\lambda_1\lambda_2-5\lambda_2^2\Big\}\nonumber\\
&+2(\omega_{22}^1+\omega_{33}^1)(47\lambda_1^3+3\lambda_1^2\lambda_2-14\lambda_1\lambda_2^2)\omega_{22}^1\nonumber\\
&+2(\omega_{22}^1+\omega_{33}^1)(20\lambda_1^3+12\lambda_1^2\lambda_2+4\lambda_1\lambda_2^2)\omega_{33}^1=0,\nonumber
\end{align}
which, together with \eqref {G1}, leads to
\begin{align}\label{G13}
&2(4\lambda_1\lambda_2^2+57\lambda_1^2\lambda_2+89\lambda_1^3)(\omega_{22}^1)^2
+2(22\lambda_1\lambda_2^2+66\lambda_1^2\lambda_2+62\lambda_1^3)(\omega_{33}^1)^2\\
=&(9\lambda_1\lambda_2^2+27\lambda_1^2\lambda_2+21\lambda_1^3)
(69\lambda_1^2+42\lambda_1\lambda_2+5\lambda_2^2)\nonumber\\
&-2\lambda_2(3\lambda_1+\lambda_2)
(-10\lambda_1\lambda_2^2+15\lambda_1^2\lambda_2+67\lambda_1^3).\nonumber
\end{align}
Equations \eqref{G1}, \eqref{G11} and \eqref{G13}  can be rewritten in the following forms:
\begin{align}
&\omega_{22}^1\omega_{33}^1= L,\label{G14} \\
&M_1(\omega_{22}^1)^2+N_1(\omega_{33}^1)^2 = K_1,\label{G15} \\
&M_2(\omega_{22}^1)^2+N_2(\omega_{33}^1)^2 = K_2,\label{G16}
\end{align}
where
\begin{align*}
    &L=3\lambda_1\lambda_2+\lambda_2^2,\\
    &M_1=2\lambda_2-2\lambda_1, \quad N_1=-8\lambda_1-2\lambda_2,\\
    &K_1=-17\lambda_1\lambda_2^2-51\lambda_1^2\lambda_2-63\lambda_1^3,\\
    &M_2=8\lambda_1\lambda_2^2+114\lambda_1^2\lambda_2+178\lambda_1^3,\\
    &N_2=44\lambda_1\lambda_2^2+132\lambda_1^2\lambda_2+124\lambda_1^3,\\
    &K_2=1449\lambda_1^5+2343\lambda_1^4\lambda_2+1636\lambda_1^3\lambda_2^2+543\lambda_1^2\lambda_2^3+65\lambda_1\lambda_2^4.
\end{align*}
Similarly to the above cases, we   eliminate the terms of
$\omega_{22}^1$ and $\omega_{33}^1$ from \eqref{G14}, \eqref{G15}
and \eqref{G16}. At last, we obtain
\begin{align}\label{G17}
    &M_1(M_1N_2-M_2N_1)^2L^2+N_1(M_1K_2-M_2K_1)^2\\
    &-K_1(M_1N_2-M_2N_1)(M_1K_2-M_2K_1)=0,\nonumber
\end{align}
which reduces to
\begin{align*}
    &-12168\lambda_1^2\lambda_2^{11}-170352\lambda_1^3\lambda_2^{10}-1361256\lambda_1^4\lambda_2^9-7125456\lambda_1^5\lambda_2^8\\
    &-26220576\lambda_1^6\lambda_2^7-68130960\lambda_1^7\lambda_2^6-119958840\lambda_1^8\lambda_2^5-122714784\lambda_1^9\lambda_2^4\\
    &-15894792\lambda_1^{10}\lambda_2^3+135723600\lambda_1^{11}\lambda_2^2+162996624\lambda_1^{12}\lambda_2+62868960\lambda_1^{13}=0.
\end{align*}
Therefore, we get a polynomial equation concerning $k=\frac {\lambda_1}{\lambda_2}$ with constant coefficients as follows:
\begin{align*}
    &62868960k^{13}+162996624k^{12}
    +135723600k^{11}-15894792k^{10}\\
    &-122714784k^9-119958840k^8
    -68130960k^7-26220576k^6\\
    &-7125456k^5-1361256k^4
    -170352k^3-12168k^2=0.
\end{align*}
The above equation shows that $k=\frac {\lambda_1}{\lambda_2}$ must be a constant. Applying $\lambda_2=\frac {\lambda_1}{k}$ and \eqref{G4} to \eqref{L2}, it gives
\begin{align}
    \omega_{22}^1&=\frac {e_1(\lambda_2)}{\lambda_2-\lambda_1}=\frac {e_1(\lambda_1)}{(1-k)\lambda_1},\label{G18}\\
    \omega_{33}^1&=\frac {e_1(\lambda_3)}{\lambda_3-\lambda_1}=\frac {(3k+1)e_1(\lambda_1)}{(4k+1)\lambda_1}.\label{G19}
\end{align}
It follows from \eqref{G18} and \eqref{G19} that \eqref{L1} and
\eqref{G1} become
\begin{align}
    e_1e_1(\lambda_1)&=\frac {(3k^2-6k-2)e_1^2(\lambda_1)}{(4k+1)(k-1)\lambda_1}+\frac {(10k^2+6k+2)\lambda_1^3}{k^2},\label{G20}\\
    e_1^2(\lambda_1)&=\frac {(1-k)(4k+1)}{k^2}\lambda_1^4.\label{G21}
\end{align}
Taking into account $e_1(\lambda_1)\neq 0$ and differentiating \eqref{G21} yield
\begin{equation}\label{G22}
    e_1e_1(\lambda_1)=\frac {2(1-k)(4k+1)}{k^2}\lambda_1^3.
\end{equation}
Substituting \eqref{G21}-\eqref{G22} into \eqref{G20} gives
\begin{equation*}
    15k^2+6k+2=0.
\end{equation*}
Obviously, the above quadratic equation has no real root, which
contradicts to our assumption. Therefore, we complete a proof
of Theorem 1.1. \end{proof}

\begin{remark}
In \cite{Gupta2016}, the authors claimed that they proved the
minimality of biharmonic hypersurfaces in $\mathbb R^5$.
Unfortunately, there is a crucial error in their proofs. In Page 11 of
\cite{Gupta2016}: Since $a_1=-b_1 =-c_1$, the equation (3.43) in \cite{Gupta2016} should be
corrected as
\begin{align}\label{U1}
-4H(\lambda_4-\lambda_3)(2\lambda_2-\lambda_3-\lambda_4)a_1+b_2+b_3-c_2-c_3=k_2-k_3.
\end{align}
 Combining (3.41), (3.42) with
\eqref{U1} in \cite{Gupta2016},  one has
\begin{align}\label{U2}
4H\Big\{(\lambda_3-\lambda_2)^2-(\lambda_4-\lambda_2)^2-(\lambda_4-\lambda_3)(2\lambda_2-\lambda_3-
\lambda_4)\Big\} a_1=0.
\end{align}
However, \eqref{U2} gives no further information due to the fact that
\begin{align}
(\lambda_3-\lambda_2)^2-(\lambda_4-\lambda_2)^2-(\lambda_4-\lambda_3)(2\lambda_2-\lambda_3-
\lambda_4)\equiv0. \nonumber
\end{align}
In fact, the expression $k_1-k_2+k_3-k_1+k_2-k_3\equiv0$ must yield
no further information. Those arguments in \cite{Gupta2016} are no longer useful for
proceeding with Chen's conjecture.
\end{remark}

\medskip



\end{document}